\documentclass[reqno,11pt]{amsart}

\usepackage{amsmath,amssymb}
\usepackage{amsthm}
\usepackage{aliascnt}
\usepackage{mathtools}
\usepackage{bm}
\usepackage{cleveref}

\numberwithin{equation}{section}

\newtheorem{thm}{Theorem}[section]

\newaliascnt{prop}{thm}
\newtheorem{prop}[prop]{Proposition}
\aliascntresetthe{prop}

\newaliascnt{lem}{thm}
\newtheorem{lem}[lem]{Lemma}
\aliascntresetthe{lem}

\newaliascnt{cor}{thm}

\aliascntresetthe{cor}

\theoremstyle{definition}
\newaliascnt{dfn}{thm}
\newtheorem{dfn}[dfn]{Definition}
\aliascntresetthe{dfn}

\newaliascnt{ex}{thm}

\aliascntresetthe{ex}

\newaliascnt{rem}{thm}
\newtheorem{rem}[rem]{Remark}
\aliascntresetthe{rem}

\crefname{thm}{Theorem}{Theorems}
\crefname{prop}{Proposition}{Propositions}
\crefname{lem}{Lemma}{Lemmas}
\crefname{cor}{Corollary}{Corollaries}
\crefname{dfn}{Definition}{Definitions}
\crefname{ex}{Example}{Examples}
\crefname{rem}{Remark}{Remarks}

\Crefname{thm}{Theorem}{Theorems}
\Crefname{prop}{Proposition}{Propositions}
\Crefname{lem}{Lemma}{Lemmas}
\Crefname{cor}{Corollary}{Corollaries}
\Crefname{dfn}{Definition}{Definitions}
\Crefname{ex}{Example}{Examples}
\Crefname{rem}{Remark}{Remarks}

\begin{document}

\title[contact discontinuities for 2-D compressible Euler system]{Non-uniqueness of admissible weak solutions to the two-dimensional barotropic compressible Euler system with contact discontinuities} 
\author{Kotaro Horimoto}
\address{Graduate School of Science, Tohoku University, Aoba, Sendai 980-8578, Japan}
\email{horimoto.kotaro.p6@dc.tohoku.ac.jp}

\date{\today}

\maketitle

\begin{abstract}
This paper is concerned with the Riemann problem for the two-dimensional barotropic compressible Euler system with a general strictly increasing pressure law. 
By means of convex integration, the existence of infinitely many admissible weak solutions is established for certain Riemann initial data for which the corresponding one-dimensional self-similar solution consists solely of a contact discontinuity.     
\end{abstract}

\section{Introduction}

The compressible Euler system constitutes a fundamental model for the motion of inviscid compressible fluids and plays a central role in the theory of hyperbolic conservation laws. 
A basic question in this theory concerns the well-posedness of solutions, in particular their existence and uniqueness. 
In this paper, we investigate this question for the barotropic compressible Euler system in the two-dimensional whole space. 
More precisely, we consider the following initial-value problem: 
\begin{align} 
    \partial_t \rho + \mathrm{div} \, \bm{m} &= 0 && \text{in} \;(0 ,\infty) \times \mathbb{R}^2, \label{eq:Euler mass} \\
    \partial_t \bm{m} + \mathrm{div} \, \frac{\bm{m} \otimes \bm{m}}{\rho}  + \nabla \, p(\rho) &= \bm{0} && \text{in} \;(0 ,\infty) \times \mathbb{R}^2, \label{eq:Euler momentum} \\
    \rho(0,\cdot) &= \rho_\mathrm{init} && \text{in} \;\mathbb{R}^2, \label{eq:Euler init1} \\
    \bm{m}(0,\cdot) &= \bm{m}_\mathrm{init} && \text{in} \;\mathbb{R}^2. \label{eq:Euler init2} 
\end{align}
The unknowns are the density $\rho = \rho(t,\bm{x}) > 0$ and the momentum $\bm{m} = \bm{m}(t,\bm{x}) \in \mathbb{R}^2$.
The first equation represents the conservation of mass, and the second one expresses the conservation of momentum. 
We prescribe initial data $\rho_\mathrm{init}\in L^\infty(\mathbb{R}^2;\mathbb{R}^+)$ and $\bm{m}_\mathrm{init}\in L^\infty(\mathbb{R}^2;\mathbb{R}^2)$, 
where we denote $\mathbb{R}^+=(0,\infty)$. 
The pressure $p$ is a given function of $\rho$, and we assume that
\begin{equation} \label{eq:p-condition}
  p \in C^1(\mathbb{R}^+;\mathbb{R}^+) \quad \text{and} \quad p'(\rho)>0 \;\; \text{for} \;\; \rho\in\mathbb{R}^+.  
\end{equation}
In particular, the commonly used polytropic pressure law $p(\rho) = \rho^\gamma$ with $\gamma \geq 1$ satisfies \eqref{eq:p-condition}.  
Hence, the present setting covers many of the pressure laws considered in the literature, see e.g.~\cite{Ch-De-Kr2015, Ch-Kr2014, Kl-Ma2018}.
In contrast to many previous works, which often focus on specific pressure laws such as the polytropic case, we do not impose any additional structural assumptions on $p$ beyond $C^1$ regularity and monotonicity.

We next introduce the notion of a weak solution to the system \eqref{eq:Euler mass}--\eqref{eq:Euler init2}. 
Since $p'(\rho)>0$ for $\rho>0$, the system is \emph{hyperbolic}. 
As is already observed in simple examples such as Burgers' equation, see Dafermos~\cite[Section~4.2]{Dafermos}, classical solutions to hyperbolic conservation laws do not exist globally in time, even for smooth initial data. 
We are therefore led to consider \emph{weak solutions}.

\begin{dfn}[Weak solutions]  \label{weak sol.}
A pair $(\rho , \bm{m})\in L^\infty((0,\infty)\times\mathbb{R}^2 ; \mathbb{R}^+\times \mathbb{R}^2)$ is called a \emph{weak solution} to the initial-value problem \eqref{eq:Euler mass}--\eqref{eq:Euler init2} 
if the following identities hold for all test functions $(\phi,\boldsymbol{\varphi}) \in C_c^\infty([0,\infty) \times \mathbb{R}^2; \mathbb{R}\times\mathbb{R}^2)$:
\begin{align*}
\int_0 ^ \infty \int_ {\mathbb{R}^2} \left( \rho \partial_t \phi + \bm{m} \cdot \nabla \phi \right) \; \mathrm{d}\bm{x}\,\mathrm{d}t + \int _{\mathbb{R}^2} \rho _\mathrm{init} \phi (0,\cdot) \; \mathrm{d}\bm{x} &= 0, \\
\int_0 ^ \infty \int_ {\mathbb{R}^2} \left( \bm{m} \cdot \partial_t \boldsymbol{\varphi} + \frac{\bm{m} \otimes \bm{m}}{\rho} : \nabla \boldsymbol{\varphi} + p(\rho)\,\mathrm{div}\, \boldsymbol{\varphi} \right) \: \mathrm{d}\bm{x} \, \mathrm{d}t \quad& \\
+ \int _{\mathbb{R}^2} \bm{m} _{\mathrm{init}}  \cdot \boldsymbol{\varphi}(0,\cdot) \; \mathrm{d}\bm{x} &= 0. 
\end{align*}
\end{dfn}

As is well known from the example of the Burgers' equation~\cite[Section~4.4]{Dafermos}, weak solutions are in general not unique. 
To overcome the lack of uniqueness, one imposes an additional admissibility (entropy) condition that rules out physically irrelevant weak solutions.
\begin{dfn}[Admissible weak solutions]\label{def:admissible sol.}
A weak solution $(\rho,\bm{m})$ is called \emph{admissible} if  
\begin{align}
\int_0 ^ \infty \int_ {\mathbb{R}^2} \left(\frac{\lvert \bm{m} \rvert ^2}{2\rho} + P(\rho)\right) \partial_t \psi 
+ \left(\frac{\lvert \bm{m} \rvert^2}{2\rho} + P(\rho) + p(\rho)\right)\frac{\bm{m}}{\rho} \cdot \nabla \psi 
\, \mathrm{d}\bm{x} \, \mathrm{d}t  \notag \\  
+ \int _{\mathbb{R}^2} \left(\frac{\lvert \bm{m} _\mathrm{init} \rvert^2}{2\rho _\mathrm{init}} + P(\rho _\mathrm{init})\right) 
\psi(0, \cdot) \; \mathrm{d}\bm{x} \,\geq\, 0 \label{eq:entropy}
\end{align}
for all non-negative test functions $\psi \in C_c ^\infty ([0,\infty)\times \mathbb{R}^2;[0,\infty))$, 
where the pressure potential $P$ is defined by 
\begin{equation}\label{eq:potential}
    P(\rho)=\rho \int_{\rho^*}^\rho \frac{p(r)}{r^2} \, \mathrm{d}r
\end{equation}
for $\rho > 0$, and the constant $\rho^* > 0$ is arbitrary.
\end{dfn}

In one space dimension, the admissibility condition ensures the uniqueness of solutions. 
More precisely, for sufficiently small initial data, Glimm~\cite{Glimm1965} constructed an admissible weak solution,  
and Bressan et al.~\cite{Bressan2000} proved that this solution is unique within the class of BV functions.

In contrast, in higher space dimensions, uniqueness of admissible weak solutions generally fails.
In seminal works, De~Lellis and Sz\'ekelyhidi~\cite{De-Sze2009, De-Sze2010} developed a \emph{convex integration} framework for the incompressible Euler system and constructed infinitely many admissible weak solutions for suitable initial data. 
Moreover, in~\cite{De-Sze2010} an analogous non-uniqueness result for the compressible system is proved. 
Subsequently, the convex integration approach has been further developed and applied extensively to both the compressible and incompressible Euler systems, 
see e.g.~\cite{Ch2014, Ch-Kr-Ma-Sch2021, Fei2014, Sze2011}.

This paper addresses ourselves to Riemann initial data of the form
\begin{equation}\label{eq:Riemann init}
(\rho _\mathrm{init} , \bm{m}_\mathrm{init})(\bm{x})= 
  \begin{dcases*} 
   (\rho_+ \,,\, \bm{m}_+) & if $y>0$, \\
   (\rho_- \,,\, \bm{m}_-) & if $y<0$,
  \end{dcases*}
\end{equation}
where $\rho_\pm>0$ and $\bm{m}_\pm \in \mathbb{R}^2$ are fixed constants.
We denote the space variables by $\bm{x}=(x,y) \in \mathbb{R}^2$. 
The system \eqref{eq:Euler mass}--\eqref{eq:Euler init2} equipped with \eqref{eq:Riemann init} is referred to as the Riemann problem. 
Such data can be viewed as one-dimensional Riemann data trivially extended to two space dimensions. 

We consider existence and uniqueness of admissible weak solutions to the Riemann problem. 
Admissible weak solutions can be obtained by considering solutions that are independent of the $x$-variable. 
In one space dimension, the Riemann problem admits self-similar admissible weak solutions consisting of a finite number of constant states separated by shocks, rarefaction waves and contact discontinuities, see~\cite[Chapters~7--9]{Dafermos}.
Such solutions are commonly referred to as \emph{one-dimensional self-similar} or simply \emph{self-similar} solutions. 
These solutions can be viewed as admissible weak solutions to the original two-dimensional problem when considered as functions on $(t,\bm{x}) \in [0,\infty)\times\mathbb{R}^2$.

Concerning uniqueness, the situation is more delicate. 
A first result on non-uniqueness was obtained by Chiodaroli--De~Lellis--Kreml~\cite{Ch-De-Kr2015} in the case of the pressure law $p(\rho)=\rho^2$. 
It was shown that for certain Riemann data with a shock wave, the Riemann problem admits infinitely many admissible weak solutions. 
In contrast, Chen--Chen~\cite{Ch-Ch2007} and Feireisl--Kreml~\cite{Fei-Kr2015} proved that Riemann data whose one-dimensional self-similar solution consists solely of rarefaction waves give rise to a unique admissible weak solution. 
The result of~\cite{Fei-Kr2015} holds for any convex and strictly increasing pressure law $p \in C^1$.
Subsequent works have been devoted to the classification of Riemann data with respect to uniqueness or non-uniqueness of admissible weak solutions, 
see e.g.~\cite{Br-Ch-Kr2018, Br-Kr-Ma2021, Ch-Kr2014, Ch-Kr2018, Kl-Ma2018energy, Kl-Ma2018, Ma2021}. 
In summary, these works imply that, at least for the polytropic pressure law $p(\rho)=\rho^\gamma$ with $\gamma\ge 1$,
non-uniqueness occurs whenever the self-similar solution contains a shock,
whereas a self-similar solution consisting only of rarefaction waves is unique in the class of admissible weak solutions.

In contrast, the situation is far less understood when the self-similar solution consists solely of a contact discontinuity. 
In this regime, the uniqueness question for admissible weak solutions remains largely open. 
In the present paper, we address this problem for a subclass of these data and establish the existence of infinitely many admissible weak solutions. 

More precisely, the self-similar solution to the initial-value problem \eqref{eq:Euler mass}--\eqref{eq:Euler init2} with \eqref{eq:Riemann init} consists of a single contact discontinuity, 
if the Riemann data have a jump only in the component of the momentum tangential to the discontinuity, while both the density and the normal component of the momentum are continuous across the interface, i.e., 
\begin{equation*}
\rho_+ = \rho_-, \quad \lbrack \bm{m}_+ \rbrack_1 \neq \lbrack \bm{m}_- \rbrack_1, \quad \lbrack \bm{m}_+ \rbrack_2 = \lbrack \bm{m}_- \rbrack_2.
\end{equation*} 
Here, $\lbrack \bm{m} \rbrack_i$ denotes the $i$-th component of $\bm{m}\in\mathbb{R}^2$. 
In addition, we impose the following symmetry condition:
\begin{equation} \label{eq:symm-condi}
\lbrack \bm{m}_+ \rbrack_1 = - \lbrack \bm{m}_- \rbrack_1 .
\end{equation}
This symmetry condition is motivated from the vortex sheet initial data considered by Sz\'ekelyhidi~\cite{Sze2011} in the \emph{incompressible} Euler system, where the initial velocity field is given by 
\begin{equation*}
\bm{v}_\mathrm{init}(\bm{x})= 
  \begin{dcases*} 
   (1,0) & if $y>0$, \\
   (-1,0) & if $y<0$.
  \end{dcases*}
\end{equation*}
Using convex integration techniques, Sz\'ekelyhidi~\cite{Sze2011} proved the existence of infinitely many admissible weak solutions starting from this initial datum. 
Contact discontinuities in the compressible Euler system can be viewed as the compressible analogue of vortex sheets in the incompressible setting.
It is therefore natural to ask whether a similar non-uniqueness phenomenon can occur in the compressible case.

Recently, Krupa--Sz\'ekelyhidi~\cite{Krupa-Sze2025} obtained a related non-uniqueness result for contact discontinuity Riemann data under the symmetry condition \eqref{eq:symm-condi}, by treating the pressure law as an additional degree of freedom. 
More precisely, in~\cite{Krupa-Sze2025}, a smooth pressure law $p$ with $p'>0$ is constructed 
such that the system \eqref{eq:Euler mass}--\eqref{eq:Euler init2} with this pressure law admits infinitely many admissible weak solutions for initial data of the form $\rho_\pm = \rho$, $\bm{m}_\pm = (\pm \rho,0)$, where $\rho > 0$ is a suitable constant.
However, it has remained unclear whether such non-uniqueness can occur for a prescribed physically relevant pressure law, such as the polytropic law $p(\rho)=\rho^\gamma$.
In the present paper, by contrast, we establish non-uniqueness for arbitrary strictly increasing pressure laws $p \in C^1$. 
In particular, this shows that the presence of shocks is not necessary for non-uniqueness of admissible weak solutions to the compressible Euler system.
More precisely, our main result reads,

\begin{thm} \label{thm:main}
  Assume that $p$ satisfies \eqref{eq:p-condition}. 
  Let $\rho_0>0$ and $u_0\neq 0$ be arbitrarily given. 
  Set $\rho_+ = \rho_- = \rho_0$, $\bm{m}_+=(\rho_0u_0 \,,\, 0)$, and $\bm{m}_-=(-\rho_0u_0 \,,\, 0)$. 
  Then there exist infinitely many admissible weak solutions to \eqref{eq:Euler mass}--\eqref{eq:Euler init2} with \eqref{eq:Riemann init}.
\end{thm}

\begin{rem} 
  The admissible weak solutions constructed in \Cref{thm:main} are genuinely two-dimensional, in the sense that they depend on both spatial variables $x$ and $y$. 
  Indeed, Theorem~1.2 in~\cite{Krupa-Sze2025} yields uniqueness of one-dimensional admissible weak solutions in the class $L^\infty$ for contact discontinuity Riemann data. 
  Consequently, the infinitely many admissible weak solutions obtained here cannot be one-dimensional and must therefore exhibit genuinely two-dimensional oscillations. 
\end{rem}

Our proof is based on the refined convex integration method developed by Markfelder~\cite{Mark2024}. 
This paper is organized as follows. 
In \Cref{sec:Preliminaries}, we briefly summarize the results of~\cite{Mark2024}, which form the basis of our approach. 
In \Cref{sec:Proof}, we provide the proof of \Cref{thm:main}.

\section{Preliminaries} \label{sec:Preliminaries}
\subsection{Definitions}
In this section, we recall the convex integration approach developed in~\cite{Mark2024} that will be used in the subsequent analysis. 
In the original approaches (see e.g.~\cite{Ch-De-Kr2015, Ma2021}), the existence of infinitely many admissible weak solutions is reduced to the construction of a suitable subsolution, 
while the entropy inequalities are verified separately from the convex integration framework. 
Therefore, one of main difficulties is to additionally check the entropy inequalities \eqref{eq:entropy}. 
Markfelder~\cite{Mark2024} developed a refined convex integration method that incorporates the entropy inequalities into the framework itself. 
This idea significantly simplifies the analysis of the entropy inequality and, consequently, facilitates the construction of admissible fan subsolutions. 

To formulate this framework precisely, we first introduce the basic notation. 
In what follows, we write
\begin{equation*}
\mathcal{PH} := \mathbb{R}^2\times\mathrm{Sym}_0(2,\mathbb{R})\times\mathbb{R}\times\mathbb{R}^2,
\end{equation*}
where $\mathrm{Sym}_0(2,\mathbb{R})$ denotes the space of symmetric traceless $2\times 2$ real matrices. 
We then recall the notion of a \emph{fan partition}.

\begin{dfn}[Fan partition, see~{\cite[Definition~4.1]{Mark2024}}] \label{def:fan-partition}
  Let $\mu_0 < \mu_1 < \mu_2 < \mu_3$ be real numbers. 
  A \emph{fan partition} of the space-time domain $(0,\infty)\times\mathbb{R}^2$ is a collection of five open sets 
  $\Gamma_-, \Gamma_1, \Gamma_2, \Gamma_3, \Gamma_+$ defined by 
 \begin{align*}
   \Gamma_- &= \{ (t,\bm{x}) \in (0,\infty)\times \mathbb{R}^2 \mid y < \mu_0 t \}, \\
   \Gamma_i &= \{ (t,\bm{x}) \in (0,\infty)\times \mathbb{R}^2 \mid \mu_{i-1}t < y < \mu_i t \} \quad \text{for $i=1,2,3$,}\\
   \Gamma_+ &= \{ (t,\bm{x}) \in (0,\infty)\times \mathbb{R}^2 \mid y > \mu_3 t \}.
 \end{align*}
\end{dfn}

\begin{rem} \label{rem:5-1}
The number of sets $\Gamma_i$ in \Cref{def:fan-partition} is chosen specifically to accommodate the proof of \Cref{thm:main}. 
In the original convex integration frameworks (see e.g.~\cite{Ch-De-Kr2015, Ma2021}), the space-time domain $(0,\infty)\times\mathbb{R}^2$ is decomposed into three regions, $\Gamma_-, \Gamma_1,$ and $\Gamma_+$, 
and convex integration is applied in the single intermediate region $\Gamma_1$ to generate infinitely many solutions. 
For the class of Riemann data considered in \Cref{thm:main}, however, three- or four-region partitions are not sufficient. 
In fact, they do not allow the construction of an admissible fan subsolution satisfying all required constraints. 
This necessity leads us to introduce a finer partition consisting of five regions. 
See \Cref{rem:5-2} for further discussion. 
\end{rem}

Next, we recall the definition of an \emph{admissible fan subsolution}. 

\begin{dfn}[Admissible fan subsolution, see~{\cite[Definition~4.3]{Mark2024}}]\label{def:adf}
 An \emph{admissible fan subsolution} to the initial-value problem for the compressible Euler system \eqref{eq:Euler mass}--\eqref{eq:Euler init2} with \eqref{eq:Riemann init} is a tuple 
 \begin{equation*}
   (\rho, \overline{\bm{m}}, \overline{\mathbb{U}}, \overline{q}, \overline{\bm{F}}) \in L^\infty ((0,\infty) \times \mathbb{R}^2 ; \mathbb{R}^+\times\mathcal{PH})
 \end{equation*}
 of piecewise constant functions which satisfies the following properties: 
 \begin{enumerate}
    \item[(i)] There exist a fan partition $\Gamma_-, \Gamma_1, \Gamma_2, \Gamma_3, \Gamma_+$ of $(0,\infty)\times \mathbb{R}^2$ and constants 
                \[(\rho_i,\bm{m}_i,\mathbb{U}_i,q_i,\bm{F}_i) \in \mathbb{R}^+\times\mathcal{PH} \qquad \text{for $i=1,2,3$} \] 
                such that 
                \begin{equation*}
                  (\rho, \overline{\bm{m}}, \overline{\mathbb{U}}, \overline{q}, \overline{\bm{F}})= 
                  \begin{dcases*} 
                     (\rho_-, \bm{m}_-, \mathbb{U}_-, q_-, \bm{F}_-) & if $(t,\bm{x}) \in \Gamma_-$, \\
                     (\rho_i, \bm{m}_i, \mathbb{U}_i, q_i, \bm{F}_i) & if $(t,\bm{x}) \in \Gamma_i$, \; $i=1,2,3$, \\
                     (\rho_+, \bm{m}_+, \mathbb{U}_+, q_+, \bm{F}_+) & if $(t,\bm{x}) \in \Gamma_+$,
                  \end{dcases*}
               \end{equation*}
               where 
               \begin{align}
                  q_\pm &:= \frac{|\bm{m}_\pm|^2}{2\rho_\pm} + p(\rho_\pm), \label{eq:q_pm} \\
                  \mathbb{U}_\pm &:= \frac{\bm{m}_\pm \otimes \bm{m}_\pm}{\rho_\pm} + (p(\rho_\pm)-q_\pm)\mathbb{I}_2, \label{eq:U_pm} \\
                  \bm{F}_\pm &:= \frac{q_\pm + P(\rho_\pm)}{\rho_\pm}\bm{m}_\pm. \label{eq:F_pm}
               \end{align}
    \item[(ii)] It holds that 
                 \begin{equation*}
                  \lambda_{\rm{max}} \left( \frac{\bm{m}_i \otimes \bm{m}_i}{\rho_i} - \mathbb{U}_i + (p(\rho_i) - q_i)\mathbb{I}_2 \right) < 0 \quad \text{for $i=1,2,3$.}
                 \end{equation*}
    \item[(iii)] The equations 
                  \begin{align*}
                    \int_0 ^ \infty \int_ {\mathbb{R}^2} \left( \rho \partial_t \phi + \overline{\bm{m}} \cdot \nabla \phi \right) \; \mathrm{d}\bm{x}\,\mathrm{d}t + \int _{\mathbb{R}^2} \rho _\mathrm{init} \phi (0,\cdot) \; \mathrm{d}\bm{x} &= 0, \\
                    \int_0 ^ \infty \int_ {\mathbb{R}^2} \left( \overline{\bm{m}} \cdot \partial_t \boldsymbol{\varphi} + \overline{\mathbb{U}} : \nabla \boldsymbol{\varphi} + \overline{q}\,\mathrm{div}\, \boldsymbol{\varphi} \right) \: \mathrm{d}\bm{x} \, \mathrm{d}t + \int _{\mathbb{R}^2} \bm{m} _{\mathrm{init}}  \cdot \boldsymbol{\varphi}(0,\cdot) \; \mathrm{d}\bm{x} &= 0, 
                  \end{align*}
                  and the inequality 
                  \begin{align*} 
                    \int_0 ^ \infty \int_ {\mathbb{R}^2} \left(\overline{q} + P(\rho) - p(\rho)\right) \, \partial_t \psi \, + \, \overline{\bm{F}} \cdot \nabla \psi \;\; \mathrm{d}\bm{x} \, \mathrm{d}t \\ 
                    + \int _{\mathbb{R}^2} \left(\frac{\lvert \bm{m} _\mathrm{init} \rvert^2}{2\rho _\mathrm{init}} + P(\rho _\mathrm{init})\right) \psi(0, \cdot) \; \mathrm{d}\bm{x} \geq 0
                  \end{align*}
                  hold for all test functions $(\phi, \boldsymbol{\varphi}, \psi) \in C_c^\infty([0,\infty) \times \mathbb{R}^2; \mathbb{R}\times\mathbb{R}^2 \times \mathbb{R})$ with $\psi \geq 0$.
 \end{enumerate}
\end{dfn}

\begin{rem}
The notion of admissible fan subsolution introduced in \Cref{def:adf} plays a crucial role in the framework of~\cite{Mark2024}. 
In the previous frameworks as in~\cite{Ch-De-Kr2015, Ma2021}, the energy flux $\overline{\bm{F}}$ is prescribed as $\overline{\bm{F}} = \frac{\overline{q}+P(\rho)}{\rho}\overline{\bm{m}}$. 
In contrast, the formulation in \Cref{def:adf} allows us to treat the flux $\overline{\bm{F}}$ as an unknown. 
This additional flexibility simplifies the treatment of the entropy inequalities.
\end{rem}

\subsection{Sufficient condition for non-uniqueness}
We now describe the convex integration method developed in~\cite{Mark2024}, which asserts that the existence of an admissible fan subsolution implies the existence of infinitely many admissible weak solutions.

\begin{prop}[cf.~{\cite[Proposition~4.5]{Mark2024}}] \label{prop:reduction}
 Assume that $p \in C^1(\mathbb{R}^+;\mathbb{R}^+)$. 
 Let $(\rho_\pm, \, \bm{m}_\pm) \in \mathbb{R}^+ \times \mathbb{R}^2$ be initial states such that there exists an admissible fan subsolution $(\rho, \overline{\bm{m}}, \overline{\mathbb{U}}, \overline{q}, \overline{\bm{F}})$ to the initial-value problem \eqref{eq:Euler mass}--\eqref{eq:Euler init2} with \eqref{eq:Riemann init}. 
 Then there exist infinitely many functions $\bm{m} \in L^\infty((0,\infty)\times\mathbb{R}^2;\mathbb{R}^2)$ such that 
 each pair $(\rho,\bm{m})$ is an admissible weak solution to the initial-value problem \eqref{eq:Euler mass}--\eqref{eq:Euler init2} with \eqref{eq:Riemann init}. 
\end{prop}

\begin{rem}
We slightly adapt Proposition~4.5 in~\cite{Mark2024} by relaxing the regularity assumption on $p$.
In that work, the pressure law is assumed to belong to $C^1([0,\infty))$. 
Since the construction only involves states with density uniformly bounded away from zero, the argument does not rely on the behavior at $\rho=0$. 
Moreover, since the convex integration method in~\cite{Mark2024} does not require monotonicity of $p$, we only assume that $p \in C^1((0,\infty))$ in \Cref{prop:reduction}, whereas monotonicity is needed in the proof of \Cref{thm:main}. 
\end{rem}

As shown in~\cite{Mark2024}, 
the definition of an admissible fan subsolution (see \Cref{def:adf}) can be reduced to a system of algebraic equations and inequalities. 

\begin{prop}[see~{\cite[Proposition~4.6]{Mark2024}}]\label{prop:large-system}
Let $\rho_\pm > 0$ and $\bm{m}_\pm \in \mathbb{R}^2$ be given. 
Assume that there exist constants $\mu_0, \mu_1,\mu_2,\mu_3 \in \mathbb{R}$ and $(\rho_i,\bm{m}_i,\mathbb{U}_i,q_i,\bm{F}_i) \in \mathbb{R}^+\times\mathcal{PH}$ 
for $i=1,2,3$ which fulfill the following algebraic equations and inequalities\/{\rm{:}}
\begin{itemize}
   \item Order of the speeds\/{\rm{:}}
     \begin{equation}\label{eq:speed}
      \mu_0 < \mu_1 < \mu_2 <\mu_3.
     \end{equation}

   \item Rankine--Hugoniot conditions at each interface\/{\rm{:}}
    \begin{align}
       &\mu_i(\rho_i-\rho_{i+1})=\lbrack \bm{m}_i \rbrack_2 - \lbrack \bm{m}_{i+1} \rbrack_2, \label{eq:RH mass} \\
       &\mu_i(\lbrack \bm{m}_i \rbrack_1 - \lbrack \bm{m}_{i+1} \rbrack_1) = \lbrack \mathbb{U}_i \rbrack_{12} - \lbrack \mathbb{U}_{i+1} \rbrack_{12}, \label{eq:RH m1} \\
       &\mu_i(\lbrack \bm{m}_i \rbrack_2 - \lbrack \bm{m}_{i+1} \rbrack_2) = - \lbrack \mathbb{U}_i \rbrack_{11} + q_i + \lbrack \mathbb{U}_{i+1} \rbrack_{11} - q_{i+1}, \label{eq:RH m2} \\
       &\mu_i(q_i + P(\rho_i) - p(\rho_i) - q_{i+1} - P(\rho_{i+1}) + p(\rho_{i+1})) \leq \lbrack \bm{F}_i \rbrack_2 - \lbrack \bm{F}_{i+1} \rbrack_2 \label{eq:RH en}
      \end{align}
   for $i=0,1,2,3$, where 
    \begin{align*}
      (\rho_0,\bm{m}_0,\mathbb{U}_0,q_0,\bm{F}_0) &= (\rho_-,\bm{m}_-,\mathbb{U}_-,q_-,\bm{F}_-),  \\
      (\rho_4,\bm{m}_4,\mathbb{U}_4,q_4,\bm{F}_4) &= (\rho_+,\bm{m}_+,\mathbb{U}_+,q_+,\bm{F}_+) 
    \end{align*}
   with $q_\pm, \mathbb{U}_\pm, \bm{F}_\pm$ defined in \eqref{eq:q_pm}--\eqref{eq:F_pm}.
   
   \item Subsolution conditions\/{\rm{:}}
   \begin{align}
      &\frac{\lbrack \bm{m}_i \rbrack_1^2 + \lbrack \bm{m}_i \rbrack_2^2}{\rho_i} + 2(p(\rho_i)-q_i) < 0, \label{eq:subsol tr} \\
      &\left(\frac{\lbrack \bm{m}_i \rbrack_1^2}{\rho_i} - \lbrack \mathbb{U}_i \rbrack_{11} + p(\rho_i) - q_i \right) \left(\frac{\lbrack \bm{m}_i \rbrack_2^2}{\rho_i} + \lbrack \mathbb{U}_i \rbrack_{11} + p(\rho_i) - q_i \right) \notag \\
      &\qquad - \left(\frac{\lbrack \bm{m}_i \rbrack_1 \lbrack \bm{m}_i \rbrack_2}{\rho_i} - \lbrack \mathbb{U}_i \rbrack _{12} \right)^2 > 0 \label{eq:subsol det}
   \end{align}
   for $i=1,2,3$.
\end{itemize}
Then $(\rho_i,\bm{m}_i,\mathbb{U}_i,q_i,\bm{F}_i)_{i=1,2,3}$ define an admissible fan subsolution to the initial-value problem \eqref{eq:Euler mass}--\eqref{eq:Euler init2} with \eqref{eq:Riemann init}, 
where the corresponding fan partition is determined by $\mu_i$, $i=0,1,2,3$.
\end{prop}

As a consequence of \Cref{prop:reduction,prop:large-system}, 
it suffices to construct constants satisfying \eqref{eq:speed}--\eqref{eq:subsol det} to prove \Cref{thm:main}.

\section{Proof of \Cref{thm:main}} \label{sec:Proof}
In this section, we prove \Cref{thm:main} by finding constants satisfying the algebraic system \eqref{eq:speed}--\eqref{eq:subsol det}. 
To this end, let $\rho_0 > 0$ and $u_0 \neq 0$ be given, and assume that $p \in C^1(\mathbb{R}^+;\mathbb{R}^+)$ is strictly increasing.  
Motivated from the symmetry of the initial data, we construct an admissible fan subsolution exhibiting the same symmetry. 
In particular, we may assume that  
\begin{align} 
  (\rho_1,\bm{m}_1,\mathbb{U}_1,q_1,\bm{F}_1) &= (\rho_3,-\bm{m}_3,\mathbb{U}_3,q_3,-\bm{F}_3), \label{eq:symm} \\
  (\mu_0, \mu_1, \mu_2, \mu_3) &= (-b, -a, a, b) \label{eq:mu1-mu4}
\end{align}
for some constants $0<a<b$. 

The following lemma shows that, under the symmetry assumptions \eqref{eq:symm} and \eqref{eq:mu1-mu4},
the Rankine--Hugoniot conditions can be simplified to explicit algebraic relations,
which will be used throughout the subsequent analysis.

\begin{lem}[Reduction of the Rankine--Hugoniot conditions] \label{lem:Reduction}
  Under the symmetry assumptions \eqref{eq:symm} and \eqref{eq:mu1-mu4}, 
  the Rankine--Hugoniot conditions \eqref{eq:RH mass}--\eqref{eq:RH m2} is reduced to the following algebraic relations\/{\rm{:}}
  \begin{equation} \label{eq:Reduction}
    \lbrack \bm{m}_1 \rbrack_2 = -b(\rho_1-\rho_0), \quad
    \bm{m}_2 = \bm{0}, \quad
    \rho_2 = \rho_1 - \frac{b}{a}\bigl( \rho_1 - \rho_0 \bigr).
  \end{equation}
  Moreover, the matrices $\mathbb{U}_i$ can be expressed explicitly in terms of the parameters $(a,b,\rho_1,\lbrack\bm{m}_1\rbrack_1,q_1,q_2)$ as
  \begin{align}
    \lbrack \mathbb{U}_1 \rbrack_{11} &= -b^2(\rho_1-\rho_0) - p(\rho_0) + q_1, \label{eq:U1}\\
    \lbrack \mathbb{U}_1 \rbrack_{12} &= -b(\rho_0u_0 + \lbrack \bm{m}_1 \rbrack_1), \label{eq:U2}\\
    \lbrack \mathbb{U}_2 \rbrack_{11} &= -b(b-a)(\rho_1-\rho_0) - p(\rho_0) + q_2, \label{eq:U3}\\
    \lbrack \mathbb{U}_2 \rbrack_{12} &= -(b-a) \lbrack \bm{m}_1 \rbrack_1 - b \rho_0 u_0. \label{eq:U4}
  \end{align}
\end{lem}

\begin{proof}
  Since we consider the Riemann initial data, 
  \begin{equation*}
    \begin{aligned}
      \rho_- &= \rho_0, & \rho_+ &= \rho_0, \\ 
      \bm{m}_- &= \begin{pmatrix} -\rho_0u_0 \\ 0 \end{pmatrix}, & 
      \bm{m}_+ &= \begin{pmatrix} \rho_0u_0 \\ 0 \end{pmatrix},  
    \end{aligned}
  \end{equation*}
  it follows from \eqref{eq:q_pm} and \eqref{eq:U_pm} that 
  \begin{equation*}
   q_\pm = \frac{1}{2}\rho_0u_0^2 + p(\rho_0), \quad
   \mathbb{U}_\pm = 
    \begin{pmatrix}
      \frac{1}{2}\rho_0u_0^2 & 0 \\
      0 & - \frac{1}{2}\rho_0u_0^2 \\
    \end{pmatrix}.
   \end{equation*} 
  Applying the Rankine--Hugoniot conditions \eqref{eq:RH mass}--\eqref{eq:RH m2} across each interface, we obtain
  \begin{itemize} 
  \item Rankine--Hugoniot conditions across the leftmost interface $y=-bt$:
    \begin{align*}
      -b(\rho_0-\rho_1) &= - \lbrack \bm{m}_1 \rbrack_2, \\ 
      -b(-\rho_0u_0 - \lbrack \bm{m}_1 \rbrack_1) &= - \lbrack \mathbb{U}_1 \rbrack_{12}, \\ 
      -b(0 - \lbrack \bm{m}_1 \rbrack_2) &= p(\rho_0) + \lbrack \mathbb{U}_1 \rbrack_{11} - q_1.
    \end{align*}
  \item Rankine--Hugoniot conditions across the interface $y=-at$:
    \begin{align*}
      -a(\rho_1-\rho_2) &= \lbrack \bm{m}_1 \rbrack_2 - \lbrack \bm{m}_2 \rbrack_2, \\
      -a(\lbrack \bm{m}_1 \rbrack_1 - \lbrack \bm{m}_2 \rbrack_1) &= \lbrack \mathbb{U}_1 \rbrack_{12} - \lbrack \mathbb{U}_2 \rbrack_{12}, \\
      -a(\lbrack \bm{m}_1 \rbrack_2 - \lbrack \bm{m}_2 \rbrack_2) &= - \lbrack \mathbb{U}_1 \rbrack_{11} + q_1 + \lbrack \mathbb{U}_2 \rbrack_{11} - q_2.
    \end{align*}
  \item Rankine--Hugoniot conditions across the interface $y=at$:
    \begin{align*}
      a(\rho_2-\rho_1) &= \lbrack \bm{m}_2 \rbrack_2 + \lbrack \bm{m}_1 \rbrack_2, \\
      a(\lbrack \bm{m}_2 \rbrack_1 + \lbrack \bm{m}_1 \rbrack_1) &= \lbrack \mathbb{U}_2 \rbrack_{12} - \lbrack \mathbb{U}_1 \rbrack_{12}, \\
      a(\lbrack \bm{m}_2 \rbrack_2 + \lbrack \bm{m}_1 \rbrack_2) &= - \lbrack \mathbb{U}_2 \rbrack_{11} + q_2 + \lbrack \mathbb{U}_1 \rbrack_{11} - q_1.
    \end{align*}
  \item Rankine--Hugoniot conditions across the rightmost interface $y=bt$:
    \begin{align*}
      b(\rho_1-\rho_0) &= - \lbrack \bm{m}_1 \rbrack_2, \\
      b(- \lbrack \bm{m}_1 \rbrack_1 - \rho_0u_0) &= \lbrack \mathbb{U}_1 \rbrack_{12}, \\
      b(- \lbrack \bm{m}_1 \rbrack_2 - 0) &= - \lbrack \mathbb{U}_1 \rbrack_{11} + q_1 - p(\rho_0).
    \end{align*}
  \end{itemize}
  Solving the above system of algebraic equations, we obtain 
  \begin{equation*} 
    \lbrack \bm{m}_1 \rbrack_2 = -b(\rho_1-\rho_0), \quad
    \bm{m}_2 = \bm{0}, \quad
    \rho_2 = \rho_1 - \frac{b}{a}\bigl( \rho_1 - \rho_0 \bigr).
  \end{equation*}
  Moreover, the matrices $\mathbb{U}_i$ can be written explicitly as \eqref{eq:U1}--\eqref{eq:U4}. 
  This completes the proof of \Cref{lem:Reduction}.
\end{proof}

\begin{rem} \label{rem:rho1}
  As mentioned in \Cref{prop:large-system}, it is necessary to verify $\rho_i>0$ for $i=1,2,3$. 
  Since we have $\rho_2=\rho_1 - \frac{b}{a}\bigl( \rho_1 - \rho_0 \bigr)$ and $\rho_3=\rho_1$, 
  a necessary condition for $\rho_i>0$ for $i=1,2,3$ is $\rho_1 \in (0,\frac{b}{b-a}\rho_0)$.
\end{rem}

Having reduced the Rankine--Hugoniot conditions to explicit algebraic relations, 
we now turn to the entropy inequalities. 
The following lemma shows that, for the choice of energy fluxes \eqref{eq:F1} and \eqref{eq:F2} below, 
the admissibility conditions can be reduced to a single algebraic inequality.

\begin{lem}[Entropy inequalities] \label{lem:Entropy}
  For $i=1,2$, let $\bm{F}_i$ be fluxes whose second components are defined by  
  \begin{align} 
    \lbrack \bm{F}_1 \rbrack_2 &:= -b \Bigl(q_1 + P(\rho_1) - p(\rho_1) - \frac{1}{2}\rho_0u_0^2 - P(\rho_0) \Bigr), \label{eq:F1}\\
    \lbrack \bm{F}_2 \rbrack_2 &:= 0. \label{eq:F2}
  \end{align}
  Then the entropy inequalities \eqref{eq:RH en} are satisfied across all interfaces for $i=0,1,2,3$, provided that 
  \begin{align}
    b \Bigl(q_1 + P(\rho_1) &- p(\rho_1) - \frac{1}{2}\rho_0u_0^2 - P(\rho_0) \Bigr) \notag \\
    &\leq a\Bigl(q_1 + P(\rho_1) - p(\rho_1) - q_2 - P(\rho_2) + p(\rho_2) \Bigr). \label{eq:en-nokori}
  \end{align}
\end{lem}

Note that only the second component of the fluxes is relevant for the
algebraic system \eqref{eq:speed}--\eqref{eq:subsol det}. The first component can therefore be chosen
arbitrarily.

\begin{proof}
  It follows from \eqref{eq:F_pm} that
  \begin{equation*}
    \bm{F}_- = \Bigl(\tfrac{1}{2}\rho_0u_0^2 + p(\rho_0) + P(\rho_0) \Bigr) 
                 \begin{pmatrix} -u_0 \\ 0 \end{pmatrix}, \,
    \bm{F}_+ = \Bigl(\tfrac{1}{2}\rho_0u_0^2 + p(\rho_0) + P(\rho_0) \Bigr)
                 \begin{pmatrix} u_0 \\ 0 \end{pmatrix}.
  \end{equation*}
  By use of the symmetry assumptions \eqref{eq:symm} and \eqref{eq:mu1-mu4}, the entropy inequalities \eqref{eq:RH en} can be reduced to 
  \begin{align*}
    -b \Bigl(\frac{1}{2}\rho_0u_0^2 + P(\rho_0) - q_1 - P(\rho_1) + p(\rho_1) \Bigr) &\leq - \lbrack \bm{F}_1 \rbrack_2, \\
    -a(q_1 + P(\rho_1) - p(\rho_1) - q_2 - P(\rho_2) + p(\rho_2)) &\leq \lbrack \bm{F}_1 \rbrack_2 - \lbrack \bm{F}_2 \rbrack_2, \\
    a(q_2 + P(\rho_2) - p(\rho_2) - q_1 - P(\rho_1) + p(\rho_1)) &\leq \lbrack \bm{F}_2 \rbrack_2 + \lbrack \bm{F}_1 \rbrack_2, \\
    b \Bigl(q_1 + P(\rho_1) - p(\rho_1) - \frac{1}{2}\rho_0u_0^2 - P(\rho_0) \Bigr) &\leq - \lbrack \bm{F}_1 \rbrack_2.
  \end{align*}
  We choose $\lbrack \bm{F}_1 \rbrack_2$ so that the first and the last inequalities hold with equality. 
  More precisely, we set  
  \begin{equation*} 
    \lbrack \bm{F}_1 \rbrack_2 := -b \Bigl(q_1 + P(\rho_1) - p(\rho_1) - \frac{1}{2}\rho_0u_0^2 - P(\rho_0) \Bigr). 
  \end{equation*}
  With this choice, the second and the third inequalities imply that
  \begin{align}
    \lbrack \bm{F}_2 \rbrack_2 &\leq \phantom{-}a(q_1 + P(\rho_1) - p(\rho_1) - q_2 - P(\rho_2) + p(\rho_2)) + \lbrack \bm{F}_1 \rbrack_2, \label{eq:F-ue} \\
    \lbrack \bm{F}_2 \rbrack_2 &\geq -a(q_1 + P(\rho_1) - p(\rho_1) - q_2 - P(\rho_2) + p(\rho_2)) - \lbrack \bm{F}_1 \rbrack_2. \label{eq:F-shita}
    \end{align}
  Therefore, if we set
  \begin{equation*} 
    \lbrack \bm{F}_2 \rbrack_2 := 0, 
  \end{equation*}
  then the inequality \eqref{eq:F-ue} (now equivalent to \eqref{eq:F-shita}) is satisfied, provided that   
  \begin{align*}
    b \Bigl(q_1 + P(\rho_1) &- p(\rho_1) - \frac{1}{2}\rho_0u_0^2 - P(\rho_0) \Bigr) \\
    &\leq a\Bigl(q_1 + P(\rho_1) - p(\rho_1) - q_2 - P(\rho_2) + p(\rho_2) \Bigr). 
  \end{align*}
  This completes the proof of \Cref{lem:Entropy}.
\end{proof}

We are now in a position to verify the subsolution conditions. 
The following lemma shows that under the symmetry assumptions, 
the subsolution conditions \eqref{eq:subsol tr} and \eqref{eq:subsol det} are compatible with the entropy condition \eqref{eq:en-nokori}. 
Combined with \cref{prop:reduction,prop:large-system}, this yields the existence 
of infinitely many admissible weak solutions.

\begin{lem}[Subsolution conditions] \label{lem:Subsol}
  Under the symmetry assumptions \eqref{eq:symm} and \eqref{eq:mu1-mu4}, 
  there exist parameters $(a,b,\rho_1,\lbrack\bm{m}_1\rbrack_1,q_1,q_2)$ 
  satisfying the subsolution conditions \eqref{eq:subsol tr} and \eqref{eq:subsol det} together with \eqref{eq:en-nokori}. 
\end{lem}

\begin{proof}
  By substituting \eqref{eq:Reduction}--\eqref{eq:U4} into the subsolution conditions \eqref{eq:subsol tr} and \eqref{eq:subsol det}, 
  we simplify the conditions for each interface to   
  \begin{itemize}
   \item Subsolution conditions for $i=1,3$:
    \begin{align}
      &\frac{\lbrack \bm{m}_1 \rbrack_1^2 + b^2(\rho_1-\rho_0)^2}{\rho_1} + 2(p(\rho_1)-q_1) < 0, \label{eq:subsol tr1} \\
      &\Bigl(\frac{\lbrack \bm{m}_1 \rbrack_1^2}{\rho_1} + b^2(\rho_1-\rho_0) + p(\rho_0) + p(\rho_1) - 2q_1 \Bigr) \notag \\
      &\qquad \times 
      \Bigl(\frac{b^2(\rho_1-\rho_0)^2}{\rho_1} - b^2(\rho_1-\rho_0) - p(\rho_0) + p(\rho_1)  \Bigr) \notag \\
      &\qquad > \Bigl(-\frac{b (\rho_1-\rho_0) \lbrack \bm{m}_1 \rbrack_1}{\rho_1} + b(\rho_0u_0 + \lbrack \bm{m}_1 \rbrack_1) \Bigr)^2. \label{eq:subsol det1}
    \end{align}
   \item Subsolution conditions for $i=2$:
    \begin{align}
      &p\bigl(\rho_1 - \tfrac{b}{a}(\rho_1 - \rho_0)\bigr) -q_2 < 0, \label{eq:subsol tr2} \\
      &\Bigl[ b(b-a)(\rho_1-\rho_0) + p(\rho_0) + p\bigl(\rho_1 - \tfrac{b}{a}(\rho_1 - \rho_0)\bigr) - 2q_2 \Bigr] \notag \\
      &\qquad \times
      \Bigl[ -b(b-a)(\rho_1-\rho_0) - p(\rho_0) + p\bigl(\rho_1 - \tfrac{b}{a}(\rho_1 - \rho_0)\bigr) \Bigr] \notag \\
      &\qquad > \Bigl( (b-a) \lbrack \bm{m}_1 \rbrack_1 + b \rho_0 u_0 \Bigr) ^2. \label{eq:subsol det2}
    \end{align}
  \end{itemize}

  We show the existence of suitable parameters $(a,b,\rho_1,\lbrack\bm{m}_1\rbrack_1,q_1,q_2)$ which satisfy the desired conditions \eqref{eq:en-nokori} and \eqref{eq:subsol tr1}--\eqref{eq:subsol det2}.
  The proof proceeds by an explicit choice of the parameters. 
  A guiding principle is to choose some of the parameters for eliminating the right-hand side of 
  \eqref{eq:subsol det1} and then to choose the remaining parameters 
  so that each inequality is satisfied sequentially.

  We begin by fixing  
  \begin{equation}
    \lbrack \bm{m}_1 \rbrack_1 = - \rho_1 u_0, \label{eq:m1 kettei}
  \end{equation}
  which eliminates the right-hand side of \eqref{eq:subsol det1}. 
  Under this choice, conditions \eqref{eq:subsol tr1}--\eqref{eq:subsol det2} are reduced to  
  \begin{itemize}
   \item Subsolution conditions for $i=1,3$:
    \begin{align}
      &\rho_1u_0^2 + b^2 \frac{(\rho_1-\rho_0)^2}{\rho_1} + 2(p(\rho_1)-q_1) < 0, \label{eq:subsol-tr1-kai}\\
      &\Bigl(\rho_1u_0^2 + b^2(\rho_1-\rho_0) + p(\rho_0) + p(\rho_1) - 2q_1 \Bigr) \notag \\
      &\qquad \times
      \Bigl(-b^2(\rho_1-\rho_0)\,\frac{\rho_0}{\rho_1} - p(\rho_0) + p(\rho_1) \Bigr) > 0. \label{eq:subsol-det1-kai}
    \end{align}
   \item Subsolution conditions for $i=2$:
    \begin{align}
      &p\bigl(\rho_1 - \tfrac{b}{a}(\rho_1 - \rho_0) \bigr) -q_2 < 0, \label{eq:subsol-tr2-kai} \\
      &\Bigl[ b(b-a)(\rho_1-\rho_0) + p(\rho_0) + p\bigl( \rho_1 - \tfrac{b}{a}(\rho_1 - \rho_0) \bigr) - 2q_2 \Bigr] \notag \\
      &\qquad \times
       \Bigl[ -b(b-a)(\rho_1-\rho_0) - p(\rho_0) + p\bigl(\rho_1 - \tfrac{b}{a}(\rho_1 - \rho_0) \bigr) \Bigr] \notag \\
      &\qquad > \Bigl( b\rho_0 - (b-a)\rho_1 \Bigr)^2 u_0^2. \label{eq:subsol-det2-kai}
    \end{align}
  \end{itemize}
  We then choose $b>0$ as
  \begin{equation} \label{eq:b}
    b^2 := p'(\rho_0) + 1. 
  \end{equation}
  We next fix $\varepsilon \in (0,\rho_0)$ such that
  \begin{equation} \label{eq:epsilon}
    \varepsilon \, (u_0^2 + 2P'(\rho_0) + 3) + 2P(\rho_0-\varepsilon) - p(\rho_0-\varepsilon) - 2P(\rho_0) + p(\rho_0) < \frac{1}{2} \rho_0 u_0^2. 
  \end{equation}
  This is possible, since the inequality holds for $\varepsilon=0$ by $\rho_0 u_0^2 > 0$ and the left-hand side continuously depends on $\varepsilon$.
  Note that $\varepsilon$ depends only on the initial data $(\rho_0,u_0)$. 
  For this fixed $\varepsilon$, we define $q_2$ by
  \begin{equation} \label{eq:q2}
    2q_2 := p(\rho_0) + p(\rho_0 - \varepsilon) + \frac{1}{2} \rho_0 u_0^2. 
  \end{equation}
  Finally, we choose $a\in(0,b)$ sufficiently small so that  
  \begin{align}
    &\frac{2P(\rho_0+\frac{a\varepsilon}{b-a})-2P(\rho_0)}{\frac{a\varepsilon}{b-a}} - \frac{p(\rho_0+\frac{a\varepsilon}{b-a}) - p(\rho_0)}{\frac{a\varepsilon}{b-a}} \;<\; 2P'(\rho_0) - p'(\rho_0) + 1, \label{eq:a1} \\
    &\Bigl( \frac{p(\rho_0+\frac{a\varepsilon}{b-a})-p(\rho_0)}{\frac{a\varepsilon}{b-a}} \Bigr) \Bigl(1 + \frac{a\varepsilon}{(b-a)\rho_0} \Bigr) \;<\; p'(\rho_0) + 1, \label{eq:a2} \\
    &\Bigl( ab\varepsilon - \frac{1}{2} \rho_0 u_0^2 \Bigr) \Bigl( -ab\varepsilon -p(\rho_0) + p(\rho_0 - \varepsilon) \Bigr) \;>\; \Bigl(a u_0 (\rho_0 - \varepsilon)\Bigr)^2. \label{eq:a3} 
  \end{align}
  Then $a$ may depend only on $\varepsilon$ and $b$ as well as $(\rho_0,u_0)$. 
  We now define $\rho_1$ and $q_1$ in terms of $a,b,\varepsilon$ by
  \begin{align}
    \rho_1 &:= \rho_0 + \frac{a\varepsilon}{b-a}, \label{eq:rho1} \\
    2q_1 &:= \rho_1 u_0^2 + (\rho_1-\rho_0)(b^2+1) + p(\rho_0) + p(\rho_1). \label{eq:q1}
  \end{align}
  
  We verify that the parameters constructed above satisfy 
  \eqref{eq:en-nokori} and \eqref{eq:subsol-tr1-kai}--\eqref{eq:subsol-det2-kai}.
  Note that the choice $\varepsilon \in (0,\rho_0)$ ensures that $\rho_1$ defined in \eqref{eq:rho1} satisfies the condition $\rho_1 \in (0,\frac{b}{b-a}\rho_0)$ in \Cref{rem:rho1}.

  \textbf{Verification of \eqref{eq:en-nokori}.} 
    We apply \eqref{eq:rho1} and \eqref{eq:q1} to compute the left-hand side of \eqref{eq:en-nokori} as follows:
      \begin{align*}
        &b \Bigl(q_1 + P(\rho_1) - p(\rho_1) - \frac{1}{2} \rho_0 u_0^2 - P(\rho_0) \Bigr) \\
        &= \frac{b}{2} \, \biggl( \rho_1 u_0^2 + (\rho_1-\rho_0)(b^2+1) + p(\rho_0) + p(\rho_1) \\
        &\qquad\qquad + 2P(\rho_1) - 2p(\rho_1) - \rho_0 u_0^2 - 2P(\rho_0)\biggr) \\
        &= \frac{b}{2} \, (\rho_1-\rho_0) \, \Bigl(u_0^2 + b^2 + 1 + \frac{2P(\rho_1)-p(\rho_1)-2P(\rho_0)+p(\rho_0)}{\rho_1-\rho_0} \Bigr) \\
        &= \frac{ab\varepsilon}{2(b-a)} \, \Bigl(u_0^2 + b^2 + 1 + \frac{2P(\rho_1)-p(\rho_1)-2P(\rho_0)+p(\rho_0)}{\rho_1-\rho_0} \Bigr).
      \end{align*}
      We then compute the right-hand side of \eqref{eq:en-nokori}. 
      Note that it follows from \eqref{eq:Reduction} and \eqref{eq:rho1} that $\rho_2=\rho_0-\varepsilon$.
      By using \eqref{eq:q2} and \eqref{eq:q1}, the right-hand side of \eqref{eq:en-nokori} is given by
      \begin{align*}
        &a \Bigl(q_1 + P(\rho_1) - p(\rho_1) - q_2 - P(\rho_2) + p(\rho_2) \Bigr) \\
        &= \frac{a}{2} \,\biggl[ \rho_1 u_0^2 + (\rho_1-\rho_0)(b^2+1) + p(\rho_0) + p(\rho_1) + 2P(\rho_1) - 2p(\rho_1) \\
        &\qquad\qquad - p(\rho_0) - p(\rho_0 - \varepsilon) - \frac{1}{2} \rho_0 u_0^2 - 2P(\rho_0 - \varepsilon) + 2p(\rho_0 - \varepsilon) \biggr] \\
        &= \frac{a}{2} \,\biggl[ \frac{1}{2} \rho_0 u_0^2 + (\rho_1-\rho_0) (u_0^2 + b^2 + 1) \\
        &\qquad\qquad + 2P(\rho_1) - p(\rho_1) - 2P(\rho_0 - \varepsilon) + p(\rho_0 - \varepsilon) \biggr] \\
        &= \frac{a}{2} \,\biggl[ \frac{1}{2} \rho_0 u_0^2 + (\rho_1-\rho_0) \Bigl(u_0^2 + b^2 + 1 + \frac{2P(\rho_1) - p(\rho_1) - 2P(\rho_0) + p(\rho_0)}{\rho_1-\rho_0} \Bigr) \\
        &\qquad\qquad + 2P(\rho_0) - p(\rho_0) - 2P(\rho_0 - \varepsilon) + p(\rho_0 - \varepsilon) \biggr].
      \end{align*}
      Combining the above computations, we obtain    
      \begin{align*}
        &a \Bigl(q_1 + P(\rho_1) - p(\rho_1) - q_2 - P(\rho_2) + p(\rho_2) \Bigr) \\
        &\qquad\qquad - b \Bigl(q_1 + P(\rho_1) - p(\rho_1) - \frac{1}{2} \rho_0 u_0^2 - P(\rho_0) \Bigr) \\
        &=\; \frac{a}{2} \, \biggl[ \frac{1}{2} \rho_0 u_0^2 + 2P(\rho_0) - p(\rho_0) - 2P(\rho_0 - \varepsilon) + p(\rho_0 - \varepsilon) \\
        &\quad + (\rho_1-\rho_0 - \frac{b\varepsilon}{b-a}) \Bigl(u_0^2 + b^2 + 1 + \frac{2P(\rho_1) - p(\rho_1) - 2P(\rho_0) + p(\rho_0)}{\rho_1-\rho_0} \Bigr) \biggr] \\
        &=\; \frac{a}{2} \, \biggl[ \frac{1}{2} \rho_0 u_0^2 + 2P(\rho_0) - p(\rho_0) - 2P(\rho_0 - \varepsilon) + p(\rho_0 - \varepsilon) \\
        &\qquad\qquad - \varepsilon \, \Bigl(u_0^2 + b^2 + 1 + \frac{2P(\rho_1) - p(\rho_1) - 2P(\rho_0) + p(\rho_0)}{\rho_1-\rho_0} \Bigr) \biggr] 
        \end{align*} 
      Using \eqref{eq:b} and \eqref{eq:a1}, we obtain the following lower bound:
      \begin{align*}
        &\frac{a}{2} \, \biggl[ \frac{1}{2} \rho_0 u_0^2 + 2P(\rho_0) - p(\rho_0) - 2P(\rho_0 - \varepsilon) + p(\rho_0 - \varepsilon) \\
        &\qquad\qquad - \varepsilon \, \Bigl(u_0^2 + b^2 + 1 + \frac{2P(\rho_1) - p(\rho_1) - 2P(\rho_0) + p(\rho_0)}{\rho_1-\rho_0} \Bigr) \biggr] \\
        &>\; \frac{a}{2} \, \biggl[ \frac{1}{2} \rho_0 u_0^2 + 2P(\rho_0) - p(\rho_0) - 2P(\rho_0 - \varepsilon) + p(\rho_0 - \varepsilon) \\
        &\qquad\quad - \varepsilon \, \Bigl(u_0^2 + p'(\rho_0) + 2 + 2P'(\rho_0) -p'(\rho_0)  + 1 \Bigr) \biggr] \\
        &=\; \frac{a}{2} \, \biggl[ \frac{1}{2} \rho_0 u_0^2 + 2P(\rho_0) - p(\rho_0) - 2P(\rho_0 - \varepsilon) + p(\rho_0 - \varepsilon) \\
        &\qquad\quad - \varepsilon \bigl(u_0^2 + 2P'(\rho_0) + 3 \bigr) \biggr] \\
        &>\; 0,
      \end{align*} 
      where we have used \eqref{eq:epsilon} to derive the last inequality. 
      Therefore \eqref{eq:en-nokori} holds with strict inequality.
    
    \textbf{Verification of \eqref{eq:subsol-tr1-kai} and \eqref{eq:subsol-det1-kai}.} 
      A straightforward calculation shows that the sum of the two factors on the left-hand side of 
      \eqref{eq:subsol-det1-kai} coincides with the left-hand side of \eqref{eq:subsol-tr1-kai}. 
      Indeed, it holds that 
      \begin{align*}
        \Bigl(\rho_1 u_0^2 &+ b^2(\rho_1-\rho_0) + p(\rho_0) + p(\rho_1) - 2q_1 \Bigr) \\
        &+ \Bigl(-b^2(\rho_1-\rho_0)\,\frac{\rho_0}{\rho_1} - p(\rho_0) + p(\rho_1) \Bigr) \\
        &= \rho_1 u_0^2 + b^2 \frac{(\rho_1-\rho_0)^2}{\rho_1} + 2(p(\rho_1)-q_1). 
      \end{align*}
      Since \eqref{eq:subsol-tr1-kai} requires that the sum of the two factors is negative, while \eqref{eq:subsol-det1-kai} requires that the product is positive, 
      the two factors must have the same sign, and hence it suffices to verify that both factors on the left-hand side of \eqref{eq:subsol-det1-kai} are negative.  
      
      By \eqref{eq:rho1} and \eqref{eq:q1}, the first factor can be bounded from above as 
      \begin{align*}
        &\rho_1 u_0^2 + b^2(\rho_1-\rho_0) + p(\rho_0) + p(\rho_1) - 2q_1 \\
        &\quad = \rho_1 u_0^2 + b^2(\rho_1-\rho_0) + p(\rho_0) + p(\rho_1) \\
        &\qquad\quad - \Bigl(\rho_1 u_0^2 + (\rho_1-\rho_0)(b^2+1) + p(\rho_0) + p(\rho_1) \Bigr)\\
        &\quad = -(\rho_1-\rho_0) \\
        &\quad = -\,\frac{a\varepsilon}{b-a} \;<\; 0.
      \end{align*}
      Moreover, using \eqref{eq:b} and \eqref{eq:rho1}, one can estimate the second factor as 
      \begin{align*} 
        &-b^2(\rho_1-\rho_0)\,\frac{\rho_0}{\rho_1} - p(\rho_0) + p(\rho_1) \\
        &\quad = 
        -\,(\rho_1-\rho_0)\,\frac{\rho_0}{\rho_1}\,\Bigl( b^2 - \frac{p(\rho_1) - p(\rho_0)}{\rho_1-\rho_0} \frac{\rho_1}{\rho_0} \Bigr) \\
        &\quad =
        -\,\frac{a\varepsilon\rho_0}{(b-a)\rho_1}\,\biggl( p'(\rho_0) + 1 - \Bigl( \frac{p(\rho_0+\frac{a\varepsilon}{b-a})-p(\rho_0)}{\frac{a\varepsilon}{b-a}} \Bigr) \Bigl(1 + \frac{a\varepsilon}{(b-a)\rho_0} \Bigr) \biggr) \\
        &<0.
      \end{align*}
      Here the last inequality follows from \eqref{eq:a2}. 
      This shows that both factors on the left-hand side of \eqref{eq:subsol-det1-kai} 
      are negative, and hence \eqref{eq:subsol-tr1-kai} and \eqref{eq:subsol-det1-kai} hold.

    \textbf{Verification of \eqref{eq:subsol-tr2-kai}.} 
      By \eqref{eq:q2} and \eqref{eq:rho1}, the left-hand side of \eqref{eq:subsol-tr2-kai} is computed as 
      \begin{align*}
        &p\bigl(\rho_1 - \tfrac{b}{a}( \rho_1 - \rho_0 )\bigr) - q_2 \\
        &\quad =
        \frac{1}{2}\,\bigl(2p(\rho_0 - \varepsilon) - 2q_2 \bigr) \\
        &\quad =
        \frac{1}{2}\,\Bigl(2p(\rho_0 - \varepsilon) - p(\rho_0) - p(\rho_0 - \varepsilon) - \frac{1}{2} \rho_0 u_0^2 \Bigr) \\
        &\quad =
        \frac{1}{2}\,\Bigl(p(\rho_0 - \varepsilon) - p(\rho_0) - \frac{1}{2} \rho_0 u_0^2 \Bigr) 
        \;<\; 0.
      \end{align*}
      Here we have used $p'>0$ to derive the last inequality. 
      Therefore, \eqref{eq:subsol-tr2-kai} is satisfied.
    
    \textbf{Verification of \eqref{eq:subsol-det2-kai}.} 
      Using \eqref{eq:rho1}, one can simplify the right-hand side of \eqref{eq:subsol-det2-kai} as 
      \begin{equation*}
        \Bigl( b\rho_0 - (b-a)\rho_1 \Bigr)^2 u_0^2 
        \;=\;
        \Bigl(a u_0 (\rho_0 - \varepsilon)\Bigr)^2.
      \end{equation*}
      Moreover, by \eqref{eq:q2} and \eqref{eq:rho1}, the left-hand side of \eqref{eq:subsol-det2-kai} is reduced to 
      \begin{align*}
        &\Bigl( b(b-a)(\rho_1-\rho_0) + p(\rho_0) + p(\rho_0 - \varepsilon) - 2q_2 \Bigr) \\
        &\qquad \times
        \Bigl( -b(b-a)(\rho_1-\rho_0) - p(\rho_0) + p(\rho_0 - \varepsilon) \Bigr) \\
        &=
        \Bigl(ab\varepsilon + p(\rho_0) + p(\rho_0 - \varepsilon) - p(\rho_0) - p(\rho_0 - \varepsilon) - \frac{1}{2} \rho_0 u_0^2 \Bigr) \\
        &\qquad \times
        \Bigl(-ab\varepsilon - p(\rho_0) + p(\rho_0 - \varepsilon) \Bigr) \\
        &=
        \Bigl(ab\varepsilon - \frac{1}{2} \rho_0 u_0^2 \Bigr) 
        \Bigl(-ab\varepsilon - p(\rho_0) + p(\rho_0 - \varepsilon) \Bigr) \\
        &> 
        \Bigl(a u_0 (\rho_0 - \varepsilon)\Bigr)^2.
      \end{align*}
      The last inequality follows from \eqref{eq:a3}. 
      Since the last term coincides with the right-hand side of \eqref{eq:subsol-det2-kai}, inequality \eqref{eq:subsol-det2-kai} is now proved. 
  
    This completes the proof of \Cref{lem:Subsol}.
\end{proof}

By combining \Cref{lem:Subsol} with \Cref{prop:reduction,prop:large-system}, we obtain the existence of infinitely many admissible weak solutions to \eqref{eq:Euler mass}--\eqref{eq:Euler init2} with \eqref{eq:Riemann init}. 
The proof of \Cref{thm:main} is complete.

\begin{rem} \label{rem:5-2}
  As mentioned in \Cref{rem:5-1}, to prove \Cref{thm:main}, 
  a fan partition must consist of at least five regions. 
  We explain why this is the case. 
  Suppose to the contrary that there exists an admissible fan subsolution for a three-region partition, $\Gamma_-,\Gamma_1,\Gamma_+$. 
  Then from the Rankine--Hugoniot conditions \eqref{eq:RH mass} across the interfaces $y=\mu_0t$ and $y=\mu_1t$, 
  we obtain 
  \begin{align*}
    \mu_0(\rho_0-\rho_1) &= 0-\lbrack \bm{m}_1 \rbrack_2, \\
    \mu_1(\rho_1-\rho_0) &= \lbrack \bm{m}_1 \rbrack_2 - 0.
  \end{align*} 
  Since $\mu_0\neq\mu_1$, it follows that $\rho_1=\rho_0$ and $\lbrack \bm{m}_1 \rbrack_2 = 0$. 
  Moreover, from the Rankine--Hugoniot conditions \eqref{eq:RH m2} for $i=0,1$, we infer that  
  \begin{align*}
    \mu_0(0-\lbrack \bm{m}_1 \rbrack_2) &= p(\rho_0)+\lbrack \mathbb{U}_1 \rbrack_{11} - q_1, \\
    \mu_1(\lbrack \bm{m}_1 \rbrack_2 - 0) &= -\lbrack \mathbb{U}_1 \rbrack_{11} + q_1 -p(\rho_0).
  \end{align*}
  Substituting $\lbrack \bm{m}_1 \rbrack_2 = 0$ yields $\lbrack \mathbb{U}_1 \rbrack_{11}-q_1 = -p(\rho_0)$. 
  Then using the subsolution condition \eqref{eq:subsol det}, we obtain  
  \begin{equation*}
    \left(\frac{\lbrack \bm{m}_1 \rbrack_1^2}{\rho_0} - \lbrack \mathbb{U}_1 \rbrack_{11} + p(\rho_0) - q_1 \right) \Bigl( p(\rho_0) - p(\rho_0) \Bigr) > \Bigl(\lbrack \mathbb{U}_1 \rbrack _{12} \Bigr)^2. 
  \end{equation*}
  Since the left-hand side of the above inequality vanishes, this yields a contradiction. 
  Similarly, one can show that no admissible fan subsolution exists for a four-region partition. 
  For this reason, we employ the fan partition introduced in \Cref{def:fan-partition}. 
\end{rem}

\begin{rem}[Strategy for parameter selection] 
The purpose of this remark is to explain the guiding ideas behind the explicit
parameter choices used in the proof of \Cref{lem:Subsol}.
These considerations are purely heuristic and are included only to help the reader follow the algebraic choices.

As a first step, we define $\lbrack \bm{m}_1 \rbrack_1$ as in \eqref{eq:m1 kettei}, which eliminates the right-hand side of \eqref{eq:subsol det1}. 
With this choice, the subsolution conditions \eqref{eq:subsol tr1}--\eqref{eq:subsol det2} are reduced to \eqref{eq:subsol-tr1-kai}--\eqref{eq:subsol-det2-kai}. 

Next, we observe a structural relation between \eqref{eq:subsol-tr1-kai} and \eqref{eq:subsol-det1-kai}. 
In particular, the sum of the two factors on the left-hand side of \eqref{eq:subsol-det1-kai} coincides with the left-hand side of \eqref{eq:subsol-tr1-kai}. 
Hence, to verify \eqref{eq:subsol-tr1-kai} and \eqref{eq:subsol-det1-kai}, it suffices to check that both factors on the left-hand side of \eqref{eq:subsol-det1-kai} are negative, i.e.,  
\begin{align}
  \rho_1 u_0^2 + b^2(\rho_1-\rho_0) + p(\rho_0) + p(\rho_1) - 2q_1 \;&<\; 0, \label{eq:tr1 nokori} \\
  (\rho_1-\rho_0) \frac{\rho_0}{\rho_1} \Bigl( -b^2 + \frac{p(\rho_1) - p(\rho_0)}{\rho_1-\rho_0} \frac{\rho_1}{\rho_0} \Bigr) \;&<\; 0. \label{eq:det1 nokori}
\end{align}

Similarly, the sum of the two factors on the left-hand side of \eqref{eq:subsol-det2-kai} 
equals one half of the left-hand side of \eqref{eq:subsol-tr2-kai}. 
Consequently, both factors must be negative.
In particular, this requires that the second factor satisfies  
\begin{equation}
  -b(b-a)(\rho_1-\rho_0) - p(\rho_0) + p\bigl(\rho_1 - \tfrac{b}{a}(\rho_1 - \rho_0)\bigr) \; < \; 0. \label{eq:det2 kai}
\end{equation}
One can show that \eqref{eq:det2 kai} holds, provided that   
\begin{equation*}
\rho_0 < \rho_1 < \frac{b}{b-a}\rho_0.
\end{equation*}
Note that the upper bound on $\rho_1$ is imposed from \Cref{rem:rho1}. 
This motivates the definition  
\begin{equation}
  \rho_1 = \rho_0 + \frac{a}{b-a}\,\varepsilon, \label{eq:rho1 def}
\end{equation} 
where the constant $\varepsilon \in (0,\rho_0)$ will be specified later. 
In order to ensure \eqref{eq:tr1 nokori}, we define $q_1$ by
\begin{equation} 
  2q_1 = \rho_1 u_0^2 + b^2(\rho_1-\rho_0) + p(\rho_0) + p(\rho_1) + C_1(\rho_1-\rho_0), 
\end{equation}
where $C_1>0$ is a constant to be fixed later. 
Moreover, in order to verify \eqref{eq:det1 nokori}, it suffices to check that  
\begin{equation*}
  b^2 > \frac{p(\rho_1) - p(\rho_0)}{\rho_1-\rho_0} \frac{\rho_1}{\rho_0}.
\end{equation*}
Passing the limit as $a\to0$, i.e., $\rho_1 \to \rho_0$, in the above inequality yields $b^2 > p'(\rho_0)$.
Motivated from this observation, we set $b>0$ by
\begin{equation} \label{eq:b-def}
  b^2 = p'(\rho_0) + C_2
\end{equation}
for some constant $C_2>0$. 
Choosing $a>0$ sufficiently small, one can verify \eqref{eq:det1 nokori}.

We next turn to the conditions \eqref{eq:subsol-tr2-kai} and \eqref{eq:subsol-det2-kai}. 
Using \eqref{eq:rho1 def}, we deduce from \eqref{eq:subsol-tr2-kai} and \eqref{eq:subsol-det2-kai} that 
\begin{align}
  &q_2 > p(\rho_0 - \varepsilon), \label{eq:tr2 kai kai}\\
  &\Bigl(ab\varepsilon + p(\rho_0) + p(\rho_0 - \varepsilon) - 2q_2\Bigr) \Bigl(-ab\varepsilon - p(\rho_0) + p(\rho_0 - \varepsilon)\Bigr) \notag \\
  &\qquad > \Bigl(a u_0 (\rho_0 - \varepsilon)\Bigr)^2. \label{eq:det2 kai kai}
\end{align}
Heuristically, leting $a\to0$ in \eqref{eq:det2 kai kai}, we obtain  
\begin{equation*} 
  \Bigl(p(\rho_0) + p(\rho_0 - \varepsilon) - 2q_2\Bigr) \Bigl(- p(\rho_0) + p(\rho_0 - \varepsilon)\Bigr) > 0.
\end{equation*}
Since $p$ is increasing, the second factor is negative. 
This motivates defining $q_2$ by 
\begin{equation} \label{eq:q2 def}
  2q_2 = p(\rho_0) + p(\rho_0 - \varepsilon) + C_3 
\end{equation}
for some constant $C_3>0$.
For $a>0$ sufficiently small, one can then verify \eqref{eq:tr2 kai kai} and \eqref{eq:det2 kai kai}.  
  
Finally, we consider the entropy inequality \eqref{eq:en-nokori}. 
Using \eqref{eq:rho1 def}--\eqref{eq:b-def} and \eqref{eq:q2 def}, we deduce from \eqref{eq:en-nokori} that
\begin{align} 
  &(\rho_0 u_0^2 - C_3) + 2P(\rho_0) - p(\rho_0) - 2P(\rho_0 - \varepsilon) + p(\rho_0 - \varepsilon) \notag \\
  &- \varepsilon \Bigl(u_0^2 + p'(\rho_0) + C_1 + C_2 + \frac{2P(\rho_1) - 2P(\rho_0)}{\rho_1-\rho_0} - \frac{p(\rho_1) - p(\rho_0)}{\rho_1-\rho_0} \Bigr) \;\geq\; 0. \label{eq:en-heu}
\end{align}
Letting $a \to 0$, i.e., $\rho_1 \to \rho_0$, in \eqref{eq:en-heu}, we obtain
\begin{align*} 
  (\rho_0 u_0^2 &- C_3) - \varepsilon \bigl(u_0^2 + C_1 + C_2 + 2P'(\rho_0) \bigr) \\
  &+ 2P(\rho_0) - p(\rho_0) - 2P(\rho_0 - \varepsilon) + p(\rho_0 - \varepsilon) \;\geq\; 0.
\end{align*}
This suggests fixing the constants as follows: 
\begin{equation*}
  C_1 = C_2 = 1, \;\; C_3 = \frac{1}{2} \rho_0 u_0^2.
\end{equation*}
We then choose $\varepsilon>0$ sufficiently small such that
\begin{equation*}
  \frac{1}{2} \rho_0 u_0^2 - \varepsilon \bigl(u_0^2 + 2P'(\rho_0) + 2 \bigr) + 2P(\rho_0) - p(\rho_0) - 2P(\rho_0 - \varepsilon) + p(\rho_0 - \varepsilon) \;\geq\; 0.
\end{equation*}
This concludes the heuristic motivation for the parameter selection.
\end{rem}

\section*{Acknowledgements}
The author would like to thank his supervisor Professor Goro Akagi 
for his valuable comments and many helpful discussions.

\end{document}